\documentclass[a4paper,12pt]{article}
\usepackage{a4}
\usepackage{amsmath}
\usepackage{pstricks}
\usepackage{color}
\usepackage{amsfonts}
\usepackage{mathrsfs}
\allowdisplaybreaks
\usepackage{amssymb}
\usepackage{mathtools}

\usepackage[square, numbers, comma, sort&compress]{natbib}
\usepackage{amsthm,enumerate}
\usepackage{graphicx}
\usepackage[colorlinks=true,citecolor=black,linkcolor=black,urlcolor=blue]{hyperref}

\def\ni{\noindent}
\makeatletter
\newtheorem{theorem}{Theorem}[section]

\newtheorem{lemma}[theorem]{Lemma}

\newtheorem{dfn}[theorem]{Definition}

\newtheorem{rmk}[theorem]{Remark}

\def\ni{\noindent}
\numberwithin{equation}{section}

\begin{document}
	\begin{center}
		{\Large \bf Distance spectra of some double join of graphs and some new families of distance equienergetic graphs}
		
		\vspace{10mm}
		
		{\large \bf Rakshith B. R.}$^{1,\ast}$,  {\large \bf B. J. Manjunatha}$^{2,3}$
		
		\vspace{9mm}

		\baselineskip=0.20in

		$^1${\it Department of Mathematics, Manipal Institute of Technology,\\ Manipal Academy of Higher Education,\\
			Manipal 576 104, India\/} \\
		{\rm E-mail:} {\tt ranmsc08@yahoo.co.in}\\[2mm]
		
		$^2${\it Department of Mathematics\\ Vidyavardhaka College of Engineering\\
			Mysuru-570 002, India\/}\\[2mm] 
		
		$^3${\it Department of Mathematics, Sri Jayachamarajendra College of Engineering,\\ JSS Science and Technology University,\\
			Mysuru–570 006, India\/} \\
		{\rm E-mail:} {\tt manjubj@sjce.ac.in}
		
		\vspace{4mm}
		
		%(Received February 02, 2022)
	\end{center}
	
	\vspace{5mm}
	
	\baselineskip=0.20in
	
	\begin{abstract}
In this paper we compute the spectrum of a special block matrix and use it to describe the distance spectra of some double join of graphs. As an application, we give several families of distance equienergetic graphs of diameter 3.   
	\end{abstract}
	\ni  \textbf{Keywords}: Double join graph, distance spectrum, distance equienergetic graphs.   \\\\
	\ni \textbf{MSC (2010)}:  05C50.
	\baselineskip=0.30in
	\section{Introduction}
	All the graphs considered here are simple, connected and undirected. Let $G$ be a graph with vertex set $V(G)=\{v_{1},v_{2},\ldots,v_{n}\}$ and edge set $E(G)=\{e_{1},e_{2},\ldots,e_{m}\}$. We denote the spectrum of the well-known adjacency matrix of $G$ by $Spec(G)=\{\lambda_{1}(G),\lambda_{2}(G),\ldots,\lambda_{n}(G)\}$, where $\lambda_{1}(G)\ge\lambda_{2}(G)\ge\cdots\ge\lambda_{n}(G)$. The edge-vertex incidence matrix of $G$ is denoted by $M(G)$ and is defined as the matrix of order $m\times n$ whose ${ij}$-th entry is 1 if $v_{j}$ is an end vertex of the edge $e_{i}$. The \textit{distance matrix} of $G$ is $\mathcal{D}(G)=(d_{ij})_{n\times n}$, where $d_{ij}$ is the distance between the vertices $v_{i}$ and $v_{j}$ in $G$. The distance matrix was introduced in the year 1971 by Graham and Pollack \cite{graham}  to study data communication problems. Studies on the eigenvalues of the distance matrix can be found in the survey article \cite{survey}.\\[2mm]  In spectral graph theory, graph operations play an important role in construction of special classes of graphs. Some of the well-known graph operations are complement, disjoint union,
	join, the NEPS (particulary the cartesian product, the direct product, the strong product and the lexicographic product), the corona product, edge (neighborhood) corona product, subdivision (edge) vertex join, etc. A survey on spectra of graphs resulting from various
	graph operations and products is done in \cite{barik}. Computation of distance spectra of some graph compositions can be found in \cite{ds1,eq1,equ2,Haritha,adiga} and therein cited references. The  \textit{subdivision graph} of $G$, denoted by $S(G)$, is obtained by inserting a new vertex into every edge of $G$. The graph $Q(G)$ is obtained from $S(G)$ by adding an edge between two new vertices whenever the corresponding edges are adjacent. The graph $R(G)$ is obtained by introducing a new vertex corresponding to every edge of $G$, and then joining the new vertex to the end vertices of the corresponding edge. The \textit{total graph} of $G$, denoted by $T(G)$, is obtained from $R(G)$ by adding an edge between two new vertices whenever the corresponding edges are adjacent. The \textit{line graph} $L(G)$ of $G$ is a graph with vertex set $E(G)$ and two vertices are adjacent if there are adjacent edges in $G$.  \\[2mm] Let the new vertices of $S(G)$ be denoted by $e_{i}$, $i=1,2,\ldots,m$. Let $H_{1}$ and $H_{2}$ be two graphs with vertex set $V(H_{1})=\{e_{1},e_{2},\ldots,e_{m}\}$ and $V(H_{2})=\{v_{1},v_{2},\ldots,v_{n}\}$, respectively. The \textit{$(H_{1}, H_{2})$-merged subdivision graph} of $G$ \cite{rajkumar} is obtained by taking one copy of $S(G)$ and adding an edge between vertices $v_{i}$ and $v_{j}$ in $S(G)$ whenever $v_{i}v_{j}\in E(H_{2})$, and also by adding an edge between the new vertices $e_{i}$ and $e_{j}$ if $e_{i}e_{j}\in E(H_{1})$. It is denoted by $[S(G)]^{H_{1}}_{H_{2}}$. Note that $[S(G)]^{\overline{K_{m}}}_{\overline{K_{n}}}\cong S(G)$, $[S(G)]^{\overline{K_{m}}}_{G}\cong R(G)$, $[S(G)]^{L(G)}_{\overline{K_{n}}}\cong Q(G)$ and $[S(G)]^{L(G)}_{G}\cong T(G)$. In \cite{rajkumar}, the authors obtained the adjacency spectra and Laplacian spectra of $[S(G)]^{H_{1}}_{H_{2}}$ for some classes of graphs $G$, $H_{1}$ and $H_{2}$. Let $F\in\{S,R,Q,T\}$. The \textit{double join} of $F(G)$ with graphs $G_{1}$ and $G_{2}$ \cite{tian} is obtained by taking one copies of $F(G)$, $G_{1}$ and $G_{2}$, and joining each vertex of $G$ in $F(G)$ with every vertices of $G_{1}$, and also by joining each new vertex of $F(G)$ with every vertices of $G_{2}$. It is denoted by $G^{F}\vee\{G_{1}^{\bullet},G_{2}^{\circ}\}$. In \cite{tian}, Tian, He and Cui obtained the Laplacian spectrum of $G^{F}\vee\{G_{1}^{\bullet},G_{2}^{\circ}\}$ when $G$ is a regular graph. 	 
	  In analogous to the definition of double join of $F(G)~(F\in \{S,T,R,G\})$ with graphs $G_{1}$ and $G_{2}$, we define double join of $(H_{1},H_{2})$-merged subdivision graph of $G$ with graphs $G_{1}$ and $G_{2}$ as follows:
	\begin{dfn}
		The double join of $(H_{1}, H_{2})$-merged subdivision graph $[S(G)]^{H_{1}}_{H_{2}}$ with the graphs $G_{1}$ and $G_{2}$ is the graph obtained by taking one copy of   $[S(G)]^{H_{1}}_{H_{2}}$, $G_{1}$ and $G_{2}$, then joining each vertex of $G$ in $[S(G)]^{H_{1}}_{H_{2}}$ with all the vertices of $G_{2}$
		and also joining each new vertex of $[S(G)]^{H_{1}}_{H_{2}}$ with all the vertices of $G_{1}$. It is denoted by $[S(G)]^{H_{1}}_{H_{2}}\vee \{G^{\circ}_{1},G^{\bullet}_{2}\}$. 
	\end{dfn}
\begin{figure}[t]
	\centering{\includegraphics[height=1in,width=1.8in]{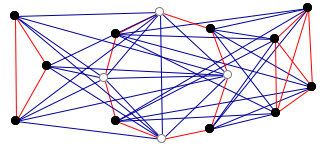}}	
	\caption{Graph $S[C_{4}]^{K_{4}}_{\overline{K_{4}}}\vee\{C_{3},K_{4}\}$}
\end{figure}
\ni This paper focuses on the distance spectrum of the
double join graph $[S(G)]^{H_{1}}_{H_{2}}\vee \{G^{\circ}_{1},G^{\bullet}_{2}\}$. Clearly the graph $[S(G)]^{H_{1}}_{H_{2}}\vee \{G^{\circ}_{1},G^{\bullet}_{2}\}$ is a generalization of the double join graph $G^{F}\vee\{G_{1}^{\bullet},G_{2}^{\circ}\}$. 
	The \textit{distance energy} of a graph $G$ is denoted by $\mathcal{E}_{D}(G)$ and is defined to be the sum of all absolute values of the eigenvalues of the distance matrix $\mathcal{D}(G)$. In analogous to graph energy (ordinary energy of a graph), the concept of distance energy was put forward in the year 2008 by Indulal, Gutman and Vijayakumar \cite{indulal}. Two graphs of same order is \textit{distance equienergetic} if their distance energies are same. In \cite{ramane}, Ramane et
	al. constructed a pair of distance equienergetic graphs of diameter 2 on $9+n$ vertices for
	all $n \ge 1$. Some other constructions of distance equienergetic graphs can be found in \cite{eq1,equ2,adiga,Haritha} and therein cited references.\\[2mm]
The paper is organized as follows. In Section 2, we compute the spectrum of a block matrix whose structure coincides with the distance matrix of  $[S(G)]^{H_{1}}_{H_{2}}\vee \{G^{\circ}_{1},G^{\bullet}_{2}\}$ in many cases. In Section 3, we give the  distance spectra of the double join graph $[S(G)]^{H_{1}}_{H_{2}}\vee \{G^{\circ}_{1},G^{\bullet}_{2}\}$ for some classes of graphs $G$, $H_{1}$, $H_{2}$, $G_{1}$ and $G_{2}$. As an application of our results, in Section 4, we give several families of distance equienergetic graphs of diameter 3.
\subsection{Notations}
 The following notations will be used in the subsequent sections.  	 	  
Let $\mathcal{M}_{m\times n}(\mathbb{R})$ denote the set of all of real matrices of order $m\times n$ and
let $\mathcal{S}_{n}(r)$ be the set of all real symmetric matrices of order $n$ such that each of its row sum is a constant $r$. We denote by  $J_{n\times m}$, the matrix of order $n\times m$ whose all entries are equal to 1. The column vector of order $n\times 1$ whose $i$th entry is 1 and all its other entries are 0, is denoted by $e_{i,n}$. Let $J^{\prime}_{n\times m}=e_{1,n}e^{T}_{1,m}$ and let $\bf{1}_{n}$ be the column vector of size $n$ whose all entries are equal to one. As usual, we denote by $C_{n}$ the cycle graph, by $K_{n}$ the complete graph, by $\overline{G}$ the complement graph of $G$, each on $n$ vertices.  \\ 
\section{Spectrum of a partitioned matrix}
This section deals with finding the spectrum of a special blocked matrix given in Definition \ref{def}.
\begin{dfn}\label{def} Let $A \in \mathcal{S}_{m}(a)$, $B\in \mathcal{S}_{n}(b)$, $C\in \mathcal{S}_{p}(c)$, $D\in \mathcal{S}_{q}(d)$ with $m\ge n$. Let $s$, $k$ and $l$ be real constants. Define a partitioned matrix $\mathcal{P}=\mathcal{P}[A,B,C,D,M,s,k,l]$ as follows:
$$\mathcal{P}=\left[\begin{array}{cccc}
A&M&sJ_{m\times p }&kJ_{m\times q}\\
M^{T}&B&kJ_{n\times p }&sJ_{n\times q}\\
sJ_{p\times m }&kJ_{p\times n}&C&lJ_{p\times q}\\
kJ_{q\times m }&sJ_{q\times n}&lJ_{q\times p}&D\\
\end{array}\right],$$
where $M$ is a real rectangular matrix of order $m\times n$ satisfying the condition (a) as given below.\begin{enumerate}[(a)]
\item
$M$ has a singular value decomposition, $M=U_{m\times m}M^{\prime}_{m\times n}V^{T}_{n\times n}$ with singular values $m_{1}(\neq 0), m_{2},\ldots, m_{n}$ such that if $X_{i}:=Ue_{i,m} (i=1,2,\ldots,m)$ and $Y_{j}:=Ve_{j,n} (j=1,2,\ldots,n)$, then $X_{i}$'s and $Y_{j}$'s  form a set of orthonormal eigenvectors of $A$ and $B$, respectively. That is, $AX_{i}=a_{i}X_{i}$ and $BY_{j}=b_{j}Y_{j}$, where $a_{i}~(i=1,2,\ldots,m)$ and $b_{j}~(j=1,2,\ldots,n)$ are the eigenvalues of $A$ and $B$, respectively. Also, $M\textbf{1}_{n}=t\textbf{1}_{m}$ and $Y_{1}=\dfrac{1}{\sqrt{n}}\textbf{1}_{n}$ for some scalar $t$.     
\end{enumerate} 
\end{dfn}
\ni The matrix $\mathcal{P}$ defined above is a real symmetric block square matrix of order $m+n+p+q$.\\[2mm] 
We need the following well-known Schur complement formula to describe the spectrum of the block matrix $\mathcal{P}$.
\begin{lemma}{\rm\cite{bapat}}\label{det}
	Let $A=\left[\begin{array}{cc}P&Q\\R&S\end{array}\right]$ be a block matrix. Let  $P$ and $S$ be square matrices.\\
	a. If $P$ is invertible, then $\det(A)=\det(P)\det(S-RP^{-1}Q)$.\\
	b. If $S$ is invertible, then  $\det(A)=\det(S)\det(P-QS^{-1}R)$.	
\end{lemma}
\ni Let the spectra of the matrices $C$ and $D$ as defined in Definition {\rm\ref{def}} be $\{c_{1}=c,c_{2},\ldots,c_{p}\}$ and $\{d_{1}=d,d_{2},\ldots,d_{q}\}$, respectively. The following theorem gives the spectrum of the block matrix $\mathcal{P}$.
\begin{theorem}\label{thm}
Let $\mathcal{P}$ be a matrix as defined in Definition {\rm\ref{def}}. Then the spectrum of the matrix $\mathcal{P}$ consists of:
\begin{enumerate}[$\bullet$]
	\item
	$c_{i}$ and $d_{j}$ for $i=2,3,\ldots,p$ and $j=2,3,\ldots,q$;
\item
$a_{i}$ for $i=n+1,n+2,\ldots, m$;
	\item
$\dfrac{1}{2}\left((a_{i}+b_{i})\pm \sqrt{(a_{i}-b_{i})^{2}+4m^{2}_{i}}\right)$ for $i=2,3,\ldots,n$;
\item
The four roots of the polynomial,\\ $f(x)={x}^{4}- \left( a+b+c+d\right) {x}^{3}+ \left((a+b)(c+d)+ab+cd-k^{2}(mq+np)-s^2(mp+nq)\right.\\\left. -{l}^{2}p
q-m_{1}^{2} \right) {x}^{2}+ \left(-cd(a+b)-ab(c+d)+s^{2}(pm(b+d)+nq(a+c))+k^{2}(np(a+d)\right.\\\left.+qm(b+c))+l^{2}pq(a+b)-2ks(lpq(m+n)+mt(p+q))+m_{1}^{2}(c+d)\right) x+npqm(s^{4}+k^{4})-s^{2}(nacq+bdpm+2lpqmt)-k^{2}(nadp+2lpqmt+bcqm)-2\,n{k}^
{2}pqm{s}^{2}-l^2(abpq-m_{1}^{2}pq)+2mkst(cq+dp)+2kpqsl(na+mb)-cdm_{1}^{2}+abcd.$	
\end{enumerate}  	
\end{theorem}
\begin{proof}
	Since $C$ and $D$ are regular real symmetric matrices, there exist orthogonal matrices $P$ and $Q$ such that $C=PC^{\prime}P^{T}$ and $D=QD^{\prime}Q^{T}$ where $C^{\prime}=diag(c_{1}=c,c_{2},\ldots,c_{p})$, $D^{\prime}=diag(d_{1}=d,d_{2},\ldots,d_{q})$, $Pe_{1,p}=\dfrac{1}{\sqrt{p}}\textbf{1}_{p}$ and $Qe_{1,q}=\dfrac{1}{\sqrt{q}}\textbf{1}_{q}$. By Definition \ref{def}, we get $M=UM^{\prime}V^{T}$, $A=UA^{\prime}U^{T}$ and $B=VB^{\prime}V^{T}$, where  $A^{\prime}=diag(a_{1},a_{2},\ldots,a_{m})$, $B^{\prime}=diag(b_{1}=b,b_{2},\ldots,b_{n})$. Since $M\textbf{1}_{n}=t\textbf{1}_{m}$, we must have, $M^{T}\textbf{1}_{m}=t_{1}\textbf{1}_{n}$ with $t_{1}=\dfrac{tm}{n}$ and also since $Y_{1}=\dfrac{1}{\sqrt{n}}\textbf{1}_{n}$, we get $X_{1}=\dfrac{MY_{1}}{m_{1}}=\dfrac{t}{m_{1}\sqrt{n}}\textbf{1}_{m}$.  Therefore, $M^{T}X_{1}=m_{1}Y_{1}$ implies  $m_{1}=\dfrac{t\sqrt{m}}{\sqrt{n}}$, and so $X_{1}=\dfrac{1}{\sqrt{m}}\textbf{1}_{m}$.\\
	Now, \begin{eqnarray*}\mathcal{P}&=&\left[\begin{array}{cccc}
	A&M&sJ_{m\times p }&kJ_{m\times q}\\
	M^{T}&B&kJ_{n\times p }&sJ_{n\times q}\\
	sJ_{p\times m }&kJ_{p\times n}&C&lJ_{p\times q}\\
	kJ_{q\times m }&sJ_{q\times n}&lJ_{q\times p}&D\\
\end{array}\right]\\[2mm]
&=&\left[\begin{array}{cccc}
	UA^{\prime}U^{T}&UM^{\prime}V^{T}&sJ_{m\times p }&kJ_{m\times q}\\
	V{M^{\prime}}^{T}U^{T}&VB^{\prime}V^{T}&kJ_{n\times p }&sJ_{n\times q}\\
	sJ_{p\times m }&kJ_{p\times n}&PC^{\prime}P^T&lJ_{p\times q}\\
	kJ_{q\times m }&sJ_{q\times n}&lJ_{q\times p}&QD^{\prime}Q^{T}\\
\end{array}\right]\\[2mm]
&=&\left[\begin{array}{cccc}
	U&0&0&0\\
	0&V&0&0\\
	0&0&P&0\\
	0&0&0&Q\\
\end{array}\right]\left[\begin{array}{cccc}
A^{\prime}&M^{\prime}&sU^{T}J_{m\times p }P&kU^{T}J_{m\times q}Q\\
{M^{\prime}}^{T}&B^{\prime}&kV^{T}J_{n\times p }P&sV^{T}J_{n\times q}Q\\
sP^{T}J_{p\times m }U&kP^{T}J_{p\times n}V&C^{\prime}&lP^{T}J_{p\times q}Q\\
kQ^{T}J_{q\times m }U&sQ^{T}J_{q\times n}V&lQ^{T}J_{q\times p}P&D^{\prime}\\
\end{array}\right]\\&&\left[\begin{array}{cccc}
U^{T}&0&0&0\\
0&V^{T}&0&0\\
0&0&P^{T}&0\\
0&0&0&Q^{T}\\
\end{array}\right]\\[2mm]
&=&\left[\begin{array}{cccc}
	U&0&0&0\\
	0&V&0&0\\
	0&0&P&0\\
	0&0&0&Q\\
\end{array}\right]\left[\begin{array}{cccc}
	A^{\prime}&M^{\prime}&s\sqrt{mp}J^{\prime}_{m\times p }&k\sqrt{mq}J^{\prime}_{m\times q}\\
	{M^{\prime}}^{T}&B^{\prime}&k\sqrt{np}J^{\prime}_{n\times p }&s\sqrt{nq}J^{\prime}_{n\times q}\\
	s\sqrt{mp}J^{\prime}_{p\times m}&k\sqrt{np}J^{\prime}_{p\times n}&C^{\prime}&l\sqrt{pq}J^{\prime}_{p\times q}\\
	k\sqrt{mq}J^{\prime}_{q\times m }&s\sqrt{nq}J^{\prime}_{q\times n}&l\sqrt{pq}J^{\prime}_{q\times p}&D^{\prime}\\
\end{array}\right]\\&&\left[\begin{array}{cccc}
	U^{T}&0&0&0\\
	0&V^{T}&0&0\\
	0&0&P^{T}&0\\
	0&0&0&Q^{T}\\
\end{array}\right].\\
\end{eqnarray*}	
Therefore,\\ $det(xI-\mathcal{P})=det\left[\begin{array}{cccc}
	xI-A^{\prime}&-M^{\prime}&-s\sqrt{mp}J^{\prime}_{m\times p }&-k\sqrt{mq}J^{\prime}_{m\times q}\\
	-{M^{\prime}}^{T}&xI-B^{\prime}&-k\sqrt{np}J^{\prime}_{n\times p }&-s\sqrt{nq}J^{\prime}_{n\times q}\\
	-s\sqrt{mp}J^{\prime}_{p\times m}&-k\sqrt{np}J^{\prime}_{p\times n}&xI-C^{\prime}&-l\sqrt{pq}J^{\prime}_{p\times q}\\
	-k\sqrt{mq}J^{\prime}_{q\times m }&-s\sqrt{nq}&-l\sqrt{pq}J^{\prime}_{q\times p}&xI-D^{\prime}\\
\end{array}\right].$\\[2mm]

\ni Expanding $det(XI-\mathcal{P})$ by Laplace's method \cite{horn} along the columns $(m+n+2),(m+n+3),\ldots,(m+n+p),(m+n+p+2),\ldots,(m+n+p+q)$, we get
$det(xI-\mathcal{P})=\\[2mm]\displaystyle\prod_{i=2}^{p}(x-c_{i})\displaystyle\prod_{j=2}^{q}(x-d_{i})det\left[\begin{array}{cccc}
	xI-A^{\prime}&-M^{\prime}&-s\sqrt{mp}e_{1,m}&-k\sqrt{mq}e_{1,m}\\
	-{M^{\prime}}^{T}&xI-B^{\prime}&-k\sqrt{np}e_{1,n}&-s\sqrt{nq}e_{1,n}\\
	-s\sqrt{mp}e^{T}_{1,m}&-k\sqrt{np}e^{T}_{1,n}&x-c&-l\sqrt{pq}\\
	-k\sqrt{mq}e^{T}_{1,m}&-s\sqrt{nq}e^{T}_{1,n}&-l\sqrt{pq}&x-d\\
\end{array}\right].$\\[2mm]
 Applying Lemma \ref{det} to the determinant on the right side of the above equation, we have
$det(xI-\mathcal{P})=\\[2mm]p_{0}(x)\displaystyle\prod_{i=2}^{p}(x-c_{i})\displaystyle\prod_{i=2}^{q}(x-d_{i})det\left[\begin{array}{cccc}
	xI-A^{\prime}-\dfrac{p_{1}(x)}{p_{0}(x)}J^{\prime}_{m\times m}&-M^{\prime}-\dfrac{p_{2}(x)}{p_{0}(x)}J^{\prime}_{m\times n}\\
	-{M^{\prime}}^{T}-\dfrac{p_{2}(x)}{p_{0}(x)}J^{\prime}_{n\times m}&xI-B^{\prime}-\dfrac{p_{3}(x)}{p_{0}(x)}J^{\prime}_{n\times n}
\end{array}\right],$\\[2mm]
where
$p_{0}(x)=(x-c)(x-d)-l^{2}pq$, $p_{1}(x)=m(s^2p(x-d)+2skpql+k^2q(x-c))$, $p_{2}(x)=\sqrt{mn}(skp(x-d)+s^{2}pql+k^{2}pql+ksq(x-c))$ and
$p_{3}(x)=n(k^{2}p(x-d)+2skpql+s^{2}q(x-c))$.\\[2mm]
Now, employing Lemma \ref{det}, we get
\begin{align*}det(xI-\mathcal{P})&=p_{0}(x)\Big(x-a_{1}-\dfrac{p_{1}(x)}{p_{0}(x)}\Big)\prod_{i=2}^{p}(x-c_{i})\prod_{i=2}^{q}(x-d_{i})\prod_{i=2}^{m}(x-a_{i})\\\times&det\Big[
xI-B^{\prime}-\dfrac{p_{3}(x)}{p_{0}(x)}J^{\prime}_{n\times n}-({M^{\prime}}^{T}+\dfrac{p_{2}(x)}{p_{0}(x)}J^{\prime}_{n\times m})(xI-A^{\prime}-\dfrac{p_{1}(x)}{p_{0}(x)}J^{\prime}_{m\times m})^{-1}\\&(M^{\prime}+\dfrac{p_{2}(x)}{p_{0}(x)}J^{\prime}_{m\times n})\Big]
\end{align*}
Upon evaluating the determinant on the right hand side of the above equation, we get\\
$det(xI-\mathcal{P})=f(x)\displaystyle\prod_{i=2}^{n}((x-a_{i})(x-b_{i})-m^{2}_{i})\displaystyle\prod_{i=2}^{p}(x-c_{i})\displaystyle\prod_{i=2}^{q}(x-d_{i})\displaystyle\prod_{i=n+1}^{m}(x-a_{i})$. This completes the proof. 
\end{proof}
\begin{rmk}
The quotient matrix \cite{Brouwer} of the partitioned matrix $\mathcal{P}$ is $\left[ \begin {array}{cccc} a&t&sp&kq\\ \noalign{\medskip}{\frac {tm}
	{n}}&b&kp&sq\\ \noalign{\medskip}sm&kn&c&lq\\ \noalign{\medskip}km&sn&
lp&d\end {array} \right]$. It can be verified that the polynomial $f(x)$ defined in Theorem {\rm\ref{thm}} is same the characteristic polynomial of the quotient matrix. 
\end{rmk}
\section{Distance spectra of some double join graph $S[G]^{H_{1}}_{H_{2}}\vee\{G^{\circ}_{1},G_{2}^{\bullet}\}$}
In this section we give the distance spectra of  $S[G]^{H_{1}}_{H_{2}}\vee\{G^{\circ}_{1},G_{2}^{\bullet}\}$ under some conditions on graphs $G$, $G_{1}$, $G_{2}$, $H_{1}$ and $H_{2}$. 
\begin{lemma}{\rm\cite{book}}
Let G be an r-regular graph with n vertices and m edges. Then\begin{enumerate}[(a)]
\item
$M(G)M^{T}(G)=L(G)+2I_{m}$ and $M^{T}(G)M(G)=A(G)+rI_{n}$.
\item
$Spec(L(G))=\{\lambda_{1}(G)+r-2,\lambda_{2}(G)+r-2,\ldots,\lambda_{n}(G)+r-2,0,0,\ldots ,0\}$.
\end{enumerate} 	
\end{lemma}
\begin{rmk}\label{rmk}
Let G be a regular graph on n vertices. If $M(G)=U\sum V^{T}$ is a singular value decomposition of $M(G)$. Then the columns of the matrix $U$ (resp. V) forms an orthonormal set of eigenvectors of the matrix $aJ_{m}+bI_{m}+cL(G)$ (resp. $aJ_{m}+bI_{m}+cA(G)$ ) for any constants $a$, $b$ and $c$. Also, we can assume that $Ve_{1,n}=\dfrac{1_{n}}{\sqrt{n}}$. Further, it may be noted that $M(G)1_{n}=21_{m}$.
\end{rmk}
\ni The following theorem gives the distance spectrum of $[S(G)]^{\overline{K_{m}}}_{\overline{K_{n}}}\vee \{G^{\circ}_{1},G^{\bullet}_{2}\}$ when $G_{1}$ and $G_{2}$ are regular graphs. 
\begin{theorem}\label{thm1}
Let G be an r-regular graph with n vertices and m edges.  Let $G_{1}$ be an $r_{1}$-regular graph of order p and let $G_{2}$ be an $r_{2}$-regular graph of order q, respectively. Then the distance spectrum of the double join graph $[S(G)]^{\overline{K_{m}}}_{\overline{K_{n}}}\vee \{G^{\circ}_{1},G^{\bullet}_{2}\}$ consists of:\begin{enumerate}
	\item
	 $-(\lambda_{i}(G_{1})+2)$ for $i=2,3,\ldots,p$ and $ d_{i}= -(\lambda_{i}(G_{2})+2)$ for $i=2,3,\ldots,q$;
	 \item
	{\rm-2} with multiplicity m-n;
	\item
	$-2\pm\sqrt{\lambda_{i}(G)+r}$ for $i=2,3,\ldots,n$;
	\item
	The four eigenvalues of the matrix 
	{\scriptsize$\left[ \begin {array}{cccc} 2(m-1)&3n-4&p&2q\\ \noalign{\medskip}3m-2r&2(n-1)&2p&q\\ \noalign{\medskip}m&2n&2(p-1)-r_{1}&3q\\ \noalign{\medskip}2m&n&
	3p&2(q-1)-r_{2}\end {array} \right]$.}
	\end{enumerate} 
\end{theorem}
\begin{proof}
 The distance matrix $\mathcal{D}$  of $[S(G)]^{\overline{K_{m}}}_{\overline{K_{n}}}\vee \{G^{\circ}_{1},G^{\bullet}_{2}\}$ is\\ \\
{\scriptsize$$\left[\begin{array}{cccc}
	2(J_{m\times m}-I_{m})&3J_{m\times n}-2M(G)&J_{m\times p }&2J_{m\times q}\\[2mm]
	3J_{n\times m}-2M^{T}(G)&2(J_{n\times n}-I_{n})&2J_{n\times p }&J_{n\times q}\\[2mm]
	J_{p\times m }&2J_{p\times n}&2(J_{p\times p}-I_{p})-A(G_{1})&3J_{p\times q}\\[2mm]
	2J_{q\times m }&J_{q\times n}&3J_{q\times p}&2(J_{q\times q}-I_{q})-A(G_{2})\\
\end{array}\right].$$}
Let $A=2(J_{m\times m}-I_{m})$, $B=2(J_{n\times n}-I_{n}),$ $C=2(J_{p\times p}-I_{p})-A(G_{1})$, $D=2(J_{q\times q}-I_{q})-A(G_{2})$, $M=3J_{m\times n}-2M(G),$ $s=1$, $k=2$ and $l=3$. Then by Remark \ref{rmk}, $\mathcal{D}=\mathcal{P}[A,B,C,D,M,s,k,l]$. Further, the eigenvalues of $A$, $B$, $C$ and $D$ are respectively,
 \begin{enumerate}
 	\item 
 	$a_{1}=2m-2$, $a_{i}=-2$ for $i=2,3,\ldots,m$.
\item
$ b_{1}=2n-2$, $b_{i}=-2$ for $i=2,3,\ldots,n$.
\item
$ c_{1}=2p-r_{1}-2$, $c_{i}=-(\lambda_{i}(G_{1})+2)$ for $i=2,3,\ldots,p$.
\item
$ d_{1}=2q-r_{2}-2$, $d_{i}=-(\lambda_{i}(G_{2})+2)$ for $i=2,3,\ldots,q$.
 \end{enumerate}
Plugging these values in Theorem \ref{thm}, we obtain the required result. 
\end{proof}
\begin{theorem}
	Let G be an r-regular triangle free graph with n vertices and m edges.  Let $G_{1}$ be an $r_{1}$-regular graph of order p and let $G_{2}$ be an $r_{2}$-regular graph of order q, respectively. If $H$ is an t-regular graph belonging to the set $\{\overline{K_{m}},L(G), \overline{L(G)}\}$, then the distance spectrum of the double join graph $[S(G)]^{H}_{\overline{G}}\vee \{G^{\circ}_{1},G^{\bullet}_{2}\}$ consists of:\begin{enumerate}
		\item
		$-(\lambda_{i}(G_{1})+2)$ for $i=2,3,\ldots,p$ and $ d_{i}= -(\lambda_{i}(G_{2})+2)$ for $i=2,3,\ldots,q$;
		\item
		{\rm$-(2+\lambda_{i}(H)$)} for $i=n+1,n+2,\ldots,m$;
		\item
		$\dfrac{1}{2}(l-k-3\pm\sqrt{(\lambda_{i}(G)+\lambda_{i}(H))^2+2\lambda_{i}(H)+6\lambda_{i}(G)+4r+1})$ for $i=2,3,\ldots,n$;
		\item
		The four eigenvalues of the matrix 
		{\scriptsize$\left[ \begin {array}{cccc} 2(m-1)-t&2(n-1)&p&2q\\ \noalign{\medskip}2m-r&n+r-1&2p&q\\ \noalign{\medskip}m&2n&2(p-1)-r_{1}&3q\\ \noalign{\medskip}2m&n&
			3p&2(q-1)-r_{2}\end {array} \right].$}
	\end{enumerate} 
\end{theorem}
\begin{proof}
	The distance matrix $\mathcal{D}$  of $[S(G)]^{H}_{\overline{G}}\vee \{G^{\circ}_{1},G^{\bullet}_{2}\}$ is\\ \\
	{\scriptsize$$\left[\begin{array}{cccc}
			2(J_{m\times m}-I_{m})-A(H)&2J_{m\times n}-M(G)&J_{m\times p }&2J_{m\times q}\\[2mm]
			2J_{n\times m}-M^{T}(G)&A(G)+J_{n\times n}-I_{n}&2J_{n\times p }&J_{n\times q}\\[2mm]
			J_{p\times m }&2J_{p\times n}&2(J_{p\times p}-I_{p})-A(G_{1})&3J_{p\times q}\\[2mm]
			2J_{q\times m }&J_{q\times n}&3J_{q\times p}&2(J_{q\times q}-I_{q})-A(G_{2})\\
		\end{array}\right].$$}
	Let $A=2(J_{m\times m}-I_{m})-A(H)$, $B=A(G)+J_{n\times n}-I_{n},$ $C=2(J_{p\times p}-I_{p})-A(G_{1})$, $D=2(J_{q\times q}-I_{q})-A(G_{2})$, $M=2J_{m\times n}-M(G),$ $s=1$, $k=2$ and $l=3$. Then by Remark \ref{rmk}, $\mathcal{D}=\mathcal{P}[A,B,C,D,M,s,k,l]$. Further, the eigenvalues of $A$, $B$, $C$ and $D$ are respectively,
	\begin{enumerate}
		\item 
		$a_{1}=2m-2-t$, $a_{i}=-(2+\lambda_{i}(H))$ for $i=2,3,\ldots,m$.
		\item
		$ b_{1}=n+r-1$, $b_{i}=\lambda_{i}(G)-1$ for $i=2,3,\ldots,n$.
		\item
		$ c_{1}=2p-r_{1}-2$, $c_{i}=-(\lambda_{i}(G_{1})+2)$ for $i=2,3,\ldots,p$.
		\item
		$ d_{1}=2q-r_{2}-2$, $d_{i}=-(\lambda_{i}(G_{2})+2)$ for $i=2,3,\ldots,q$.
	\end{enumerate}
	Plugging these values in Theorem \ref{thm}, we obtain the required result. 
\end{proof}

\begin{theorem}
	Let G be an r-regular graph with n vertices and m edges. Let $H$ be an t- regular graph belonging to the set $\{\overline{K_{n}}, K_{n}, G, \overline{G}\}$. If $G_{1}$ is an $r_{1}$-regular graph of order p and $G_{2}$ is an $r_{2}$-regular graph of order q, respectively. Then the distance spectrum of the double join graph $[S(G)]^{K_{m}}_{H}\vee \{G^{\circ}_{1},G^{\bullet}_{2}\}$ consists of:\begin{enumerate}
		\item
		$-(\lambda_{i}(G_{1})+2)$ for $i=2,3,\ldots,p$ and $ d_{i}= -(\lambda_{i}(G_{2})+2)$ for $i=2,3,\ldots,q$;
		\item
		{\rm-1} with multiplicity m-n;
		\item
		$\dfrac{1}{2}(-3-\lambda_{i}(H)\pm\sqrt {(\lambda_{i}(H)+1)^{2}+4\,\lambda_{i}(G)+4\,r})$ for $i=2,3,\ldots,n$;
		\item
		The four eigenvalues of the matrix 
		{\scriptsize$\left[ \begin {array}{cccc} m-1&2(n-1)&p&2q\\ \noalign{\medskip}{
				2m-r}&2(n-1)-t&2p&q\\ \noalign{\medskip}m&2n&2(p-1)-r_{1}&3q\\ \noalign{\medskip}2m&n&
			3p&2(q-1)-r_{2}\end {array} \right]$}
	\end{enumerate} 
\end{theorem}
\begin{proof}
	The distance matrix $\mathcal{D}$  of $[S(G)]^{K_{m}}_{H}\vee \{G^{\circ}_{1},G^{\bullet}_{2}\}$ is\\ \\
	{\scriptsize$$\left[\begin{array}{cccc}
			J_{m\times m}-I_{m}&2J_{m\times n}-M(G)&J_{m\times p }&2J_{m\times q}\\[2mm]
			2J_{n\times m}-M^{T}(G)&2(J_{n\times n}-I_{n})-A(H)&2J_{n\times p }&J_{n\times q}\\[2mm]
			J_{p\times m }&2J_{p\times n}&2(J_{p\times p}-I_{p})-A(G_{1})&3J_{p\times q}\\[2mm]
			2J_{q\times m }&J_{q\times n}&3J_{q\times p}&2(J_{q\times q}-I_{q})-A(G_{2})\\
		\end{array}\right].$$}
	Let $A=J_{m\times m}-I_{m}$, $B=2(J_{n\times n}-I_{n})-A(H),$ $C=2(J_{p\times p}-I_{p})-A(G_{1})$, $D=2(J_{q\times q}-I_{q})-A(G_{2})$, $M=2J_{m\times n}-M(G),$ $s=1$, $k=2$ and $l=3$. Then by Remark \ref{rmk}, $\mathcal{D}=\mathcal{P}[A,B,C,D,M,s,k,l]$. Further, the eigenvalues of $A$, $B$, $C$ and $D$ are respectively,
	\begin{enumerate}
		\item 
		$a_{1}=m-1$, $a_{i}=-1$ for $i=2,3,\ldots,m$.
		\item
		$ b_{1}=2n-2-t$, $b_{i}=-(\lambda_{i}(H)+2)$ for $i=2,3,\ldots,n$.
		\item
		$ c_{1}=2p-r_{1}-2$, $c_{i}=-(\lambda_{i}(G_{1})+2)$ for $i=2,3,\ldots,p$.
		\item
		$ d_{1}=2q-r_{2}-2$, $d_{i}=-(\lambda_{i}(G_{2})+2)$ for $i=2,3,\ldots,q$.
	\end{enumerate}
	Plugging these values in Theorem \ref{thm}, we obtain the required result. 
\end{proof}
\begin{theorem}
	Let G be an r-regular graph with n vertices and m edges.  Let $G_{1}$ be an $r_{1}$-regular graph of order p and let $G_{2}$ be an $r_{2}$-regular graph of order q, respectively. If $H$ is an t-regular graph belonging to the set $\{\overline{K_{m}},L(G), \overline{L(G)}\}$, then the distance spectrum of the double join graph $[S(G)]^{H}_{K_{n}}\vee \{G^{\circ}_{1},G^{\bullet}_{2}\}$ consists of:\begin{enumerate}
		\item
		$-(\lambda_{i}(G_{1})+2)$ for $i=2,3,\ldots,p$ and $ d_{i}= -(\lambda_{i}(G_{2})+2)$ for $i=2,3,\ldots,q$;
		\item
		{\rm$-(2+\lambda_{i}(H)$)} for $i=n+1,n+2,\ldots,m$;
		\item
		$\dfrac{1}{2}(-3-\lambda_{i}(H)\pm\sqrt {({\lambda_{i}(H)}+1)^{2}+4\,\lambda_{i}(G)+4\,r})$ for $i=2,3,\ldots,n$;
		\item
		The four eigenvalues of the matrix 
		{\scriptsize$\left[ \begin {array}{cccc} 2(m-1)-t&2(n-1)&p&2q\\ \noalign{\medskip}2m-r&n-1&2p&q\\ \noalign{\medskip}m&2n&2(p-1)-r_{1}&3q\\ \noalign{\medskip}2m&n&
			3p&2(q-1)-r_{2}\end {array} \right].$}
	\end{enumerate} 
\end{theorem}
\begin{proof}
	The distance matrix $\mathcal{D}$  of $[S(G)]^{H}_{K_{n}}\vee \{G^{\circ}_{1},G^{\bullet}_{2}\}$ is\\ \\
	{\scriptsize$$\left[\begin{array}{cccc}
			2(J_{m\times m}-I_{m})-A(H)&2J_{m\times n}-M(G)&J_{m\times p }&2J_{m\times q}\\[2mm]
			2J_{n\times m}-M^{T}(G)&J_{n\times n}-I_{n}&2J_{n\times p }&J_{n\times q}\\[2mm]
			J_{p\times m }&2J_{p\times n}&2(J_{p\times p}-I_{p})-A(G_{1})&3J_{p\times q}\\[2mm]
			2J_{q\times m }&J_{q\times n}&3J_{q\times p}&2(J_{q\times q}-I_{q})-A(G_{2})\\
		\end{array}\right],$$}
	where $M(G)$ is the edge-vertex incidence matrix of $G$.	Let $A=2(J_{m\times m}-I_{m})-A(H)$, $B=A(G)+J_{n\times n}-I_{n},$ $C=2(J_{p\times p}-I_{p})-A(G_{1})$, $D=2(J_{q\times q}-I_{q})-A(G_{2})$, $M=2J_{m\times n}-M(G),$ $s=1$, $k=2$ and $l=3$. Then by Remark \ref{rmk}, $\mathcal{D}=\mathcal{P}[A,B,C,D,M,s,k,l]$. Further, the eigenvalues of $A$, $B$, $C$ and $D$ are respectively,
	\begin{enumerate}
		\item 
		$a_{1}=2m-2-t$, $a_{i}=-(2+\lambda_{i}(H))$ for $i=2,3,\ldots,m$.
		\item
		$ b_{1}=n-1$, $b_{i}=-1$ for $i=2,3,\ldots,n$.
		\item
		$ c_{1}=2p-r_{1}-2$, $c_{i}=-(\lambda_{i}(G_{1})+2)$ for $i=2,3,\ldots,p$.
		\item
		$ d_{1}=2q-r_{2}-2$, $d_{i}=-(\lambda_{i}(G_{2})+2)$ for $i=2,3,\ldots,q$.
	\end{enumerate}
	Plugging these values in Theorem \ref{rmk}, we obtain the required result. 
\end{proof}
\section{Application}
As an application of our results obtained in Section 3, here we give several families of distance equienergetic graphs.
Let $\mathcal{P}_{n}$ be the set of all partitions of the positive integer $n$ into parts of size greater than or equal to 3, i.e., $\mathcal{P}_{n}=\big\{(n_{1},n_{2},\ldots, n_{k})|k\ge 1, n_{i}\ge 3 ~\text{for}~ i=1,2,\ldots, k~ \text{and}~ \displaystyle\sum_{i=1}^{k}n_{i}=n \big\}.$ Let $C(\mathcal{P}_{n})$ denote a family of disjoint union of cycles given by $C(\mathcal{P}_{n})=\Big\{\displaystyle\bigcup_{i=1}^{s} C_{n_{i}}|(n_{1},n_{2},\ldots,n_{k})\\\in \mathcal{P}_{n}\Big\}$.\\  
 Classes of distance equienergetic double join graphs are presented in the following theorem. 
\begin{theorem}
  Let $G_{1}$ be a graph in $C(\mathcal{P}_{p})$  and let $G_{2}$ be an $r_{2}$-regular graph of order q.\begin{enumerate}[(i)]
  	\item
  	If G is an r-regular graph with n vertices and m edges. Then the two classes of  double join graphs \{$[S(G)]^{\overline{K_{m}}}_{\overline{K_{n}}}\vee \{G^{\circ}_{1},G^{\bullet}_{2}\}|G_{1}\in  C(\mathcal{P}_{p})\}$ and \{$[S(G)]^{\overline{K_{m}}}_{\overline{K_{n}}}\vee \{G^{\circ}_{1},G^{\bullet}_{2}\}|G_{2}\in  C(\mathcal{P}_{q})\}$ form two families of distance equienergetic graphs.
  	\item
  	If G is an r-regular triangle free graph with n vertices and m edges and also if $H\in \{\overline{K_{m}}, L(G), \overline{L(G)}\}$. Then  the two classes of  double join graphs \{$[S(G)]^{H}_{\overline{G}}\vee \{G^{\circ}_{1},G^{\bullet}_{2}\}|G_{1}\in  C(\mathcal{P}_{p})\}$ and \{$[S(G)]^{H}_{\overline{G}}\vee \{G^{\circ}_{1},G^{\bullet}_{2}\}|G_{2}\in  C(\mathcal{P}_{q})\}$ form two families of distance equienergetic graphs.
  	\item
  If G is an r-regular graph with n vertices and m edges and if $H \in \{\overline{K_{n}}, K_{n}, G, \overline{G}\}$.  Then  the classes of  double join graphs \{$[S(G)]^{K_{m}}_{H}\vee \{G^{\circ}_{1},G^{\bullet}_{2}\}|G_{1}\in  C(\mathcal{P}_{p})\}$ and \{$[S(G)]^{K_{m}}_{H}\vee \{G^{\circ}_{1},G^{\bullet}_{2}\}|G_{2}\in  C(\mathcal{P}_{q})\}$ form  two families of distance equienergetic graphs.
\item
If G is an r-regular graph with n vertices and m edges and if $H \in \{\overline{K_{m}}, L(G), \overline{L(G)}\}$.  Then  the two classes of  double join graphs \{$[S(G)]^{H}_{K_{n}}\vee \{G^{\circ}_{1},G^{\bullet}_{2}\}|G_{1}\in  C(\mathcal{P}_{p})\}$ and \{$[S(G)]^{H}_{K_{n}}\vee \{G^{\circ}_{1},G^{\bullet}_{2}\}|G_{2}\in  C(\mathcal{P}_{q})\}$ form two families of distance equienergetic graphs.
  \end{enumerate} 
\end{theorem}
\begin{proof}
Let $G_{1}=\displaystyle\bigcup_{i=1}^{k}C_{n_{i}}$ and $G^{\prime}_{1}=\displaystyle\bigcup_{i=1}^{k^{\prime}}C_{n^{\prime}_{i}}$ be two graphs in $C(\mathcal{P}_{p})$. Then $\lambda_{1}(G_{1})=\lambda_{1}(G^{\prime}_{1})=2$, $-2\le\lambda_{i}(G_{1})$ and$-2\le \lambda_{i}(G^{\prime}_{1})$ for $i=1,2,\ldots,p$. Let $\Gamma_{1}=[S(G)]^{\overline{K_{m}}}_{\overline{K_{n}}}\vee \{G^{\circ}_{1},G^{\bullet}_{2}\}$ and $\Gamma_{2}=[S(G)]^{\overline{K_{m}}}_{\overline{K_{n}}}\vee \{{G^{\prime}}^{\circ}_{1},G^{\bullet}_{2}\}$. Then from Theorem \ref{thm1}, we get \begin{eqnarray}\label{eqn}\mathcal{E}_{D}(\Gamma_{1})-\mathcal{E}_{D}(\Gamma_{2})&=&\sum_{i=2}^{p}|\lambda_{i}(G_{1})+2|-\sum_{i=2}^{p}|\lambda_{i}(G^{\prime}_{1})+2|\notag\\
&=&\sum_{i=2}^{p}(\lambda_{i}(G_{1})+2)-\sum_{i=2}^{p}(\lambda_{i}(G^{\prime}_{1})+2).
\end{eqnarray}
Since $\sum_{i=1}^{p}\lambda_{i}(G_{1})=\sum_{i=1}^{p}\lambda_{i}(G^{\prime}_{1})=0$ and $\lambda_{1}(G_{1})=\lambda_{1}(G^{\prime}_{1})=2$, the equation  (\ref{eqn}) simplifies to $\mathcal{E}_{D}(\Gamma_{1})=\mathcal{E}_{D}(\Gamma_{2})$. This completes the proof of (i). Similarly, rest of the proof follows.
\end{proof}

\end{document}